\documentclass[reqno]{amsart}

\usepackage{amssymb}
\usepackage{amsmath}
\usepackage{color}

\newtheorem{theorem}{Theorem}[section]
\newtheorem{lemma}[theorem]{Lemma}
\newtheorem{corollary}[theorem]{Corollary}
\newtheorem{proposition}[theorem]{Proposition}

\theoremstyle{definition}
\newtheorem{definition}[theorem]{Definition}
\newtheorem{example}[theorem]{Example}

\theoremstyle{remark}
\newtheorem{remark}[theorem]{Remark}

\newcommand{\cC}{{\mathcal C}}

\newcommand{\cH}{{\mathcal H}}
\newcommand{\cI}{{\mathcal I}}

\newcommand{\cL}{{\mathcal L}}

\newcommand{\cP}{{\mathcal P}}

\newcommand{\cQ}{{\mathcal Q}}

\newcommand{\bR}{{\mathbb{R}}}
\newcommand{\bC}{{\mathbb{C}}}

\newcommand{\bN}{{\mathbb{N}}}

\newcommand{\bZ}{{\mathbb{Z}}}

\newcommand{\fB}{{\mathfrak{B}}}

\newcommand{\Tr}{\mathrm{Tr}}

\newcommand{\jed}{{\mathbb{I}}}
\newcommand{\un}{1\mkern -4mu{\rm l}}
\newcommand{\A}{{\mathcal A}}

\newcommand{\U}{{\mathcal U}}

\def\<{\langle}
\def\>{\rangle}

\def\ot{\otimes}
\def\ep{\epsilon}

\newcommand{\be}{\begin{equation}}
\newcommand{\ee}{\end{equation}}
\def\Pr{\tilde P}

\newcommand{\ben}{\begin{eqnarray}}
\newcommand{\een}{\end{eqnarray}}

\title{Bell and steering scenarios in terms of operator systems}

\author[M. Marciniak]{Marcin Marciniak}
\address{Institute of Theoretical Physics and Astrophysics, Gda{\'n}sk University,
Wita Stwosza 57, 80-952 Gda{\'n}sk, Po\-land}
\email{matmm@ug.edu.pl}

\author[M. Horodecki]{Micha{\l} Horodecki}
\address{Institute of Theoretical Physics and Astrophysics, Gda{\'n}sk University,
Wita Stwosza 57, 80-952 Gda{\'n}sk, Po\-land}
\email{fizmh@ug.edu.pl}\date{\today}

\author{Zhi Yin}
\address{School of Mathematics and Statistics, Wuhan University, Wuhan 430072, China; Institute of Theoretical Physics and Astrophysics, Gda{\'n}sk University, Wita Stwosza 57, 80-952 Gdansk, Poland}
\email{hustyinzhi@163.com}

\keywords{operator system, tensor product, corelation box}
\subjclass[2000]{46L06, 46L05, 47L25.}

\begin{document}
\maketitle
\begin{abstract}
The aim of this paper is to indicate possible applications of operator systems in qualitative description of varoius scenarios while studying non-locality. To this end we study in details the notion of generalized non-commuting cube. Following ideas of \cite{F1} and \cite{FKPT} we show in systematic way that various classes of Tsirelson's correlation boxes as well as NPA hierarchies can be described by using various operator system tensor products of generalized non-commuting cubes. Moreover, we show also that noncommuting cubes can be applied for the description of steering assemblages. 
Next we study some aproximation properties of noncommuting cubes by finite dimensional models. Finaly, we indicate possibility to use the framework operator systems for studying Bell and steering inequalities.
\end{abstract}
\baselineskip 12.5pt


\section{Operator systems and their tensor products}
In this section we briefly recall some basic notion in the theory of operator spaces and their tensor products.
\subsection{The category of operator systems}
Let $V$ be a complex vector space. We say that $V$ is a \textit{$^*$-vector space} if there is an antilinear involution $V\ni v\mapsto v^*\in V$. By $V_h$ we denote the real space of all elements satisfying $v^*=v$. By an \textit{ordered $^*$-vector space} we mean a pair $(V,V^+)$ consisting of a $^*$-vector space and a convex cone $V^+\subset V_h$ such that $V^+\cap(-V^+)=\{0\}$. Elements of $V_h$ are called \textit{hermitian elements} while elements of $V^+$ are called \textit{positive elements}.

An element $e\in V_h$ is called an \textit{order unit} if for any $v\in V_h$ there is a positive number $r$ such that $re-v\in V^+$. An order unit $e$ is called \textit{Archimedean} if $V^+$ contains all elements $v\in V$ such that $re+v\in V^+$ for every $r>0$.

We let $M_{n,m}(V)$ be the linear space of all matrices $(v_{ij})_{i=1,2,\ldots,n;\,j=1,2,\ldots,m}$ with entries in $V$. We will write $M_n(V)$ instead of $M_{n,n}(V)$. Let us notice that $M_n(V)$ can be equipped with the antilinear involution $M_n(V)\ni (v_{ij})\mapsto (v_{ji}^*)\in M_n(V)$. Hence it has a structure of a $^*$-vector space too. 

Given $n\in\bN$, let $\cP_n\subset M_n(V)_h$ be a covex cone. We say that the sequence of cones $\{\cP_n\}_{n=1}^\infty$ is a \textit{matrix ordering} on $V$ if $(M_n(V),\cP_n)$ is an ordered space for each $n\in\bN$ and $X^*\cP_nX\subset \cP_m$ for any $m,n\in\bN$ and $X\in M_{n,m}(\bC)$. The pair $(V,\{\cP_n\}_{n=1}^\infty)$ is called a \textit{matrix ordered vector space}.

An element $e\in V_h$ is called a \textit{matrix order unit} for the matrix ordered space $(V,\{\cP_n\}_{n=1}^\infty)$ if \be e_n:=\left(\begin{array}{cccc}e&&&\\&e&&\\&&\ddots&\\&&&e\end{array}\right)\ee is an order unit for $(V,\cP_n)$ for any $n\in\bN$. Respectively, it is called an \textit{Archimedean matrix order unit} for $(V,\{\cP_n\}_{n=1}^\infty)$ if $e_n$ is an Archimedean order unit for $(V,\cP_n)$ for any $n\in\bN$.

A triple $(V,\{\cP_n\}_{n=1}^\infty,e)$ is called an \textit{operator system} if $V$ is a complex $^*$-vector space, $\{\cP_n\}_{n=1}^\infty$ is a matrix ordering on $V$ and $e\in V_h$ is an Archimedean matrix order unit.

If $V,W$ are vector spaces and $\phi:V\to W$ is a linear map, then for each $n\in\bN$, we let $\phi^(n):M_n(V)\to M_n(W)$ denote a linear map given by $\phi^{(n)}((v_{ij})_{i,j=1,\ldots,n})=(\phi(v_{ij}))_{i,j=1,\ldots,n}$. If 
$(V,\{\cP_n\}_{n=1}^\infty)$ and $(W,\{\cQ_n\}_{n=1}^\infty)$ are matrix ordered spaces, a map $\phi:V\to W$ is called \textit{completely positive} if $\phi^{(n)}(\cP_n)\subset\cQ_n$ for each $n\in\bN$. We call a linear map $\phi:V\to W$ a \textit{complete order isomorphism} if $\phi$ is invertible and both $\phi$ and $\phi^{-1}$ are completely positive.

Let us consider the following example. Given a Hilbert space $\cH$ let $\fB(\cH)$ denote the $C^*$-algebra of all bounded operators on $\cH$. Let $S\subset\fB(\cH)$ be a linear subspace such that $S^*=S$ and $\jed\in S$. It is easy to observe that $S$ is a $^*$-vector space with respect to the adjoint operation. If we let $S^+=S\cap\fB(\cH)^+$ then $(S,S^+)$ has a structure of an ordered space which has $\jed$ (the identity operator) as an Archimedean unit. Moreover, $M_n(S)\subset M_n(\fB(\cH))\simeq\fB(\cH^n)$, hence $M_n(S)$ inherits an involution and order structure from $\fB(\cH^n)$ and has the $n\times n$ diagonal matrix
\be \left(\begin{array}{cccc}\jed&&&\\&\jed&&\\&&\ddots&\\&&&\jed\end{array}\right)\ee 
as an Archimedean order unit. Summing up, $S$ has a structure of an operator system. We call it a \textit{concrete operator system}.

The following theorem of Choi and Effros shows that each operator system is completely order isomorphic to some concrete operator system.
\begin{theorem}[Choi-Effros]
If $(V,\{\cP_n\}_{n=1}^\infty,e)$ is an operator system, then there exists a Hilbert space $\cH$, a concrete operator system $S\subset\fB(\cH)$, and a complete order isomorphism $\phi:V\to S$ such that $\phi(e)=\jed$.
\end{theorem}

\subsection{Duality}
Given an operator system $V$ we can consider its dual $V^\mathrm{d}$. It has a natural matrix order structure. The involution in $V^\mathrm{d}$ is defined by $\phi^*(v)=\overline{\phi(v^*)}$ for $\phi\in V^\mathrm{d}$ while the matrix order structure described by saying that $(\phi_{ij})\in M_k(V^\mathrm{d})$ is positive if and only if the map 
\be V\ni v\mapsto (\phi_{ij}(v))\in M_k(\bC)\ee 
is completely positive.
The crucial point for us is that for finite dimensional operator system $V$ the matrix oredered space $V^\mathrm{d}$ is an operator system (\cite{CE}).

\subsection{Quotients}
\label{s:quot}
If $\varphi:V\to W$ is a unital completely positive map between two operator systems and $J=\ker\varphi$, then $V/J$ has a natural matrix order structure
\be \mbox{$(s_{ij}+J)_{i,j}\in M_k(V/J)^+$ if and only if $(s_{ij})_{i,j}\in M_k(V)^+$}.\ee 
We will say that $J\subset V$ is a null-subspace if $J^*=J$ and $J$ contains no positive elements other than zero. It was shown that if $J$ is a null-subspace then $V/J$ becomes an operator system with an order unit $e+J$ (see \cite[Remark 1.2]{Kav}).
\subsection{Tensor products}
Now, let $(S,\{P_n\}_{n=1}^\infty,e)$ and $(T,\{Q_n\}_{n=1}^\infty,f)$ be two operator systems. Let $S\otimes T$ denote the algebraic tensor product of linear spaces $S$ and $T$. An \textit{operator system structure on $S\otimes T$} is a family $\tau=\{R_n\}_{n=1}^\infty$ of cones, where $R_n\subset M_n(S\otimes T)$, satisfying:
\begin{enumerate}
\item
$(S\otimes T,\{R_n\}_{n=1}^\infty,e\otimes f)$ is an operator system denoted $S\otimes_\tau T$,
\item
$P_n\otimes Q_m\subset R_{nm}$ for $n,m\in\bN$,
\item
If $\phi:S\to M_n(\bC)$ and $\psi:T\to M_m(\bC)$ are unital completely positive maps, then $\phi\otimes\psi:S\otimes_\tau T\to M_{nm}(\bC)$ is a unital completely positive map.
\end{enumerate}

The following three tensor products were introduced in \cite{FKPT} :
\begin{itemize}
\item
\textit{minimal tensor product} $S\otimes_{\mathrm{min}}T$,
\item
\textit{maximal tensor product} $S\otimes_{\mathrm{max}}T$,
\item
\textit{commuting tensor product} $S\otimes_{\mathrm{c}}T$.
\end{itemize}

Suppose that $S_1\subset T_1$ and $S_2\subset T_2$ are inclusions of operator systems and $\iota_i:S_i\to T_i$ are inclusion maps for $i=1,2$. We write $S_1\otimes_\tau S_2\subset^+ T_1\otimes_\sigma T_2$ when $\iota_1\otimes\iota_2:S_1\otimes_\tau S_2\to T_1\otimes_\sigma T_2$ is a completely positive map. If, in addition, the map $\iota_1\otimes\iota_2$ is a complete order isomorphism onto its range, then we write $S_1\otimes_\tau S_2\subset_\mathrm{coi} T_1\otimes_\sigma T_2$. 
 
\subsection{Coproduct of operator systems}
Let us recall the definition of unital free product of $C^*$-algebras. Assume that $A_1,A_2,\ldots,A_m$ are unital $C^*$-algebras. The unital free product $A_1*_1A_2*_1\ldots *_1A_n$ is a unital $C^*$-algebra with injective unital $^*$-homomorphisms $\iota_k:A_k\to A_1*_1A_2*_1\ldots *_1A_n$, $k=1,\ldots,m$, such that for any $C^*$-algebra $B$ and any unital $^*$-homorphisms $\rho_k:A_k\to B$ there is a unique unital $^*$-homomorphism $\gamma: A_1*_1A_2*_1\ldots *_1A_n\to B$ such that $\rho_k=\gamma\circ\iota_k$ for $k=1,\ldots,m$.

Given an operator system $V$ one defines its universal $C^*$-algebra $C_u^*(V)$. It has the following universality property: There is a unital complete order emebedding $\iota:V\to C_u^*(V)$ and for every unital $C^*$-algebra $B$ and every unital completely positive map $\phi:V\to B$ there is a unique unital $^*$-homomorphism $\pi:C_u^*(V)\to B$ such that $\phi=\pi\circ\iota$ (\cite{KW}).

Now, given operator systems $V_1,\ldots,V_m$ we can define their coproduct 
\begin{eqnarray}
\lefteqn{V_1\oplus_1
\ldots\oplus_1V_m=}\nonumber \\
&=&\{v_1+
\ldots+v_m:\,v_k\in V_k,\,k=1,\ldots,m\}\subset C_u^*(V_1)*_1
\ldots *_1C_u^*(V_m).
\end{eqnarray}
\begin{remark}
\label{r:univ}
When we combine universality properties of the free unital product of $C^*$-algebras and universal algebras we get the following functorial characterization of the coproduct: an operator system $U$ is a coproduct of operator systems $V_1,\ldots,V_m$ if and only of there are unital completely order embeddings $\iota_k:V_k\to U$, $k=1,\ldots,m$ and for every operator system $R$ and every unital completely positive maps $\phi_k:V_k\to R$ there is a unique unital completely positive map $\psi:U\to R$ such that $\phi_k=\psi\circ\iota_k$ (\cite{F1,Kav}).
\end{remark}

Using the similar line of argumentation as in \cite{Kav} one can easily show the following generalization of \cite[Proposition 4.7]{Kav} (see also \cite[Proposition 3.4]{F1}).
\begin{proposition}
\label{p:copr}
The operator system $V_1\oplus_1\ldots\oplus_1V_m$ is unitally completely order isomorphic to $V_1\oplus\ldots\oplus V_m/\mathcal{I}$ where
\be \cI=\left\{(\alpha_1e_1,\ldots,\alpha_me_m)\in V_1\oplus\ldots\oplus V_m:\, \alpha_k\in\bC,\,\sum_{k=1}^m\alpha_k=0\right\}.\ee 
\end{proposition}

\section{Generalized non-commuting cubes}
The aim of this section is to provide careful and detailed analysis of generalized noncommutative cubes. In \cite{FKPT} following ideas of Tsirelson the notion of \textit{noncommuting $m$-cube} was introduced. It is defined as an $(m+1)$-dimensional subspace of the universal $C^*$-algebra $C^*(h_1,\ldots,h_m)$ for $n$ noncommuting elements $h_1,\ldots,h_n$ such that $h_j^*=h_j$ and $\| h_j\|\leq 1$ for $j=1,\ldots, m$ (i.e. $h_j$ are selfadjoint contractions). The noncommuting $m$-cube $NC(m)$ is defined as the subspace
\be NC(m)=\mathrm{span}\{\jed,h_1,\ldots,h_m\}\subset C^*(h_1,\ldots,h_m)\ee
The notion of generalized non-commuting cubes appeared in \cite{chrom}.
The idea is closely related to the observation made by Tobias Fritz in \cite{F1}. He noticed that there is a possibility to relate various correlation boxes with some suitable tensor products of group $C^*$-algebras while studying non-locality. The details will be given in Section \ref{sec:bipartite_boxes}. 

Here, we will provide a more explicit construction of generalized noncommutative cubes than that of \cite{chrom}. 
We do this by combining $C^*$-algebraic construction of Fritz with  operator system approach.
We further use it to derive the relation between generalized noncommutative cubes and correlations boxes, generalizing in this way 
relation between non-commutative cubes and correlation boxes obtained in \cite{FKPT}. 
The purpose of this is to indicate some possible applications to study some approximation properties as well as violations of Bell and steering inequality.

Let us recall that if $G$ is a discrete group and $C^*(G)$ is its full $C^*$-algebra then $G$ embeds into $C^*(G)$.
Now, consider the concrete example of $G$. Let $\bZ_n$ be the cyclic with the generator $s$ of rank $n$. 
Its group $C^*$-algebra is isomorphic to $l_n^\infty$. The latter is nothing but the vector space of $n$-tuples $(z_0,z_1,\ldots,z_{n-1})$ of complex numbers with pointwise multiplication and usual involution. Let us describe the isomorphism more precisely. Let $\omega=\exp\left({2\pi i}/{n}\right)$. Consider elements $(p_a)_{a=0,1,\ldots,n-1}\in C^*(\bZ_n)$ given by the formula
\begin{equation}
\label{l:pa}
p_a=\dfrac{1}{n}\sum_{k=0}^{n-1}\omega^{ak}s^k.
\end{equation}
One can easily show that $p_a$ are orthogonal projections
\be 
\sum_{a=0}^{n-1}p_a=\jed
\ee 
and
\be 
s=\sum_{a=0}^{n-1}\omega^{-a}p_a.
\ee 
i.e. $p_a$ are spectral projections of $s$.
Let $e_0,e_1,\ldots,e_{n-1}$ be the standard basis in $l_n^\infty$. The linear map $C^*(\bZ_n)\to l_n^\infty$ given by $p_a\mapsto e_a$, $a=0,1,\ldots,n-1$ describes the promised isomorphism.

By $\bZ_n^{*m}$ we denote the free product of $m$ copies of $\bZ_n$. For $x=1,\ldots,m$ let $s_x$ be the generator of $x$-th copy of $\bZ_n$ in the free product $\bZ_n^{*m}$. They can be regarded as elements in the group $C^*$-algebra $C^*(\bZ_n^{*m})$. Define the following subspace of this algebra.
\begin{equation}
U_{m,n}=\mathrm{span}(\{\jed\}\cup\{s_x^k:\,x=1,\ldots,m,\;k=1,\ldots,n-1\}).
\end{equation}
As $U_{m,n}$ is a selfadjoint subspace of $C^*(\bZ_n^{*m})$ containing the unit it is an operator system. 

\begin{remark}\label{r:cub}
By \cite[Proposition 5.5]{FKPT} the operator system $NC(m)$ is completely order isomorphic to the system $U_{m,2}$. Thus $U_{m,n}$ can be considered as a generalization of noncommuting cube. 
\end{remark} 
\begin{remark}
\label{r:free}
Remind that for any discrete groups $G_1,\ldots,G_m$ the group $C^*$-algebra of the free product $G_1*\ldots G_m$ is nothing but the unital free products of group $C^*$-algebras, i.e.
\be C^*(G_1*\ldots*G_m)\cong C^*(G_1)*_1\ldots *_1 C^*(G_m).\ee 
When we apply it for $G_k=\bZ_n$, we get
\be C^*(\bZ_n^{*m})\cong C^*(\bZ_n)*_1\ldots *_1 C^*(\bZ_n).\ee 
Now, it follows from the construction of $U_{m,n}$ that it can be considered as a coproduct 
\be U_{m,n}=C^*(\bZ_n)\oplus_1\ldots\oplus_1 C^*(\bZ_n)\cong l_n^\infty\oplus_1\ldots\oplus_1l_n^\infty.\ee 
\end{remark}
As in (\ref{l:pa}) for every $x=1,\ldots,m$ and $a=0,1,\ldots,n-1$ we define projections
\begin{equation}
\label{l:pax}
p_{ax}=\dfrac{1}{n}\sum_{k=0}^{n-1}\omega^{ak}s_x^k.
\end{equation}
Hence $U_{m,n}$ is nothing but the span of all projections of the above form. 
\begin{proposition}
\label{p:zero}
Let $t=\sum_{a,x}z_{ax}p_{ax}\in U_{m,n}$ for some complex coefficients $z_{ax}$. Then $t=0$ if and only if there are complex complex numbers $u_x$, $x=1,\ldots,m$, such that
$\sum_{x=1}^{m}u_x=0$ and $z_{ax}=u_x$ for any pair of indices $a,x$.
\end{proposition}
\begin{proof}
Using (\ref{l:pax}) we obtain
\begin{eqnarray}
t & = &
\sum_{a=0}^{n-1}\sum_{x=1}^{m}z_{ax}\cdot\dfrac{1}{n}\sum_{k=0}^{n-1}\omega^{ak}s_x^k 
=
\dfrac{1}{n}\sum_{a=0}^{n-1}\sum_{x=1}^{m}z_{ax}\left(\jed+ \sum_{k=1}^{n-1}\omega^{ak}s_x^k \right)\nonumber \\
&=&
\dfrac{1}{n}\left[\left(\sum_{a=0}^{n-1}\sum_{x=1}^mz_{ax}\right)\jed+ \sum_{x=1}^m\sum_{k=1}^{n-1}\left(\sum_{a=0}^{n-1}\omega^{ak}z_{ax}\right) s_x^k\right]
\end{eqnarray}
Since $s_x$, $x=1,\ldots,m$, are free generators of free product $\bZ_n^{*m}$ the system $\{\jed\}\cup\{s_x^k:\,x=1,\ldots,m,\;k=1,\ldots,n-1\}$ is a linear basis in $U_{m,n}$. Thus $t=0$ implies
\begin{equation}
\label{e:ur1}
\sum_{a=0}^{n-1}\sum_{x=1}^mz_{ax}=0
\end{equation}
and
\begin{equation}
\label{e:ur2}
\sum_{a=0}^{n-1}\omega^{ak}z_{ax}=0
\end{equation}
for every $x=1,\ldots,m$ and $k=1,\ldots,n-1$. Let us fix some $x$ and loot at (\ref{e:ur2}) as a homogeneous system of $n-1$ linear equations (indexed by $k=1,\ldots,n-1$) with $n$ variables $z_{ax}$ (indexed by $a=0,1,\ldots,n-1$). Firstly, observe that if $z_{ax}$, $a=0,1,\ldots,n-1$ are equal to each other, say $z_{ax}=u_x$ for $a=0,1,\ldots,n-1$ where $u_x$ is some complex number, then they satisfy the equations (\ref{e:ur2}). Secondly, the coefficient matrix of the system has the form
\be \left[\begin{array}{cccccc}
1   & \omega   & \omega^2 & \cdots & \omega^{n-2} & \omega^{n-1} \\
1   & \omega^2 & \omega^4 & \cdots & \omega^{2n-4} & \omega^{2n-2} \\
1   & \omega^3 & \omega^6 & \cdots & \omega^{3n-6} & \omega^{3n-3} \\
\vdots & \vdots & \vdots & & \vdots & \vdots \\
1   & \omega^{n-1} & \omega^{2n-2} & \cdots & \omega^{(n-1)(n-2)} & \omega^{(n-1)^2}
\end{array}
\right]\ee 
Notice that the matrix obtained by removing the last column is a $(n-1)\times(n-1)$ dimensional Vandermonde matrix and its determinant is equal to \be \prod_{1\leq k<l\leq n-1}(\omega^k-\omega^l).\ee  
Since the last number is non-zero we conclude from Rouch{\'e}-Capelli theorem that for the fixed $x$ the set of solutions of (\ref{e:ur2}) is a 1-dimensional subspace of $l_n^\infty$. Thus, the described above solutions $z_{ax}=u_x$, $a=0,1,\ldots,n-1$, are the only solutions of the system (\ref{e:ur2}) for the fixed $x$.

Now, assume that $z_{ax}$ are such that $z_{ax}=u_x$, where $u_x$, $x=1,\ldots, m$ are some complex numbers. It follows from (\ref{e:ur1}) that 
$\sum_{x=1}^mu_x=0$ and the proof is finished.
\end{proof}

A universality property of $U_{m,n}$ for quantum measurements is a consequence of the fact mentioned in Remark \ref{r:free}.
\begin{proposition}
\label{p:univ}
Let $H$ be a Hilbert space. Assume that $(E_x^a)_{x=1,\ldots,m}^{a=0,\ldots,n-1}$ is a fa\-mi\-ly of positive operators acting on $H$ such that $\sum_{a=0}^{n-1}E_x^a=\jed$ for any $x=1,\ldots,m$. Then there is a unique unital completely positive map $\varphi:U_{m,n}\to B(H)$ such that 
$\varphi(p_{ax})=E_x^a$ for every pair $x,a$.
\end{proposition}
\begin{proof}
Given $x=1,\ldots,m$ let $\varphi_x:C^*(\bZ_n)\to B(H)$ be a linear map such that $\phi(p_a)=E_x^a$ for $a=0,1,\ldots,n-1$. Obviously it is positive and unital. Complete positivity of $\varphi_x$ follows from the fact that $C^*(\bZ_n)$ is a commutative $C^*$-algebra. Now, from universality property of the coproduct (see Remark \ref{r:univ}) it follows that there is the unique unital completely positive map $\varphi:U_{m,n}\to B(H)$ which extends each $\varphi_x$.
\end{proof}
Our next goal is to describe the operator system dual $U_{m,n}^\mathrm{d}$.

Elements of $l_{mn}^\infty$ will be described in the matrix-style as sequences with double indices: $(z_{ax})$, $a=0,1\ldots,n-1$, $x=1,\ldots,m$. Assume that a linear subspace $J\subset l_{mn}^\infty$ is composed of elements $(z_{ax})$ satisfying the condition described in the preceeding proposition, i.e.
\be 
(z_{ax})\in J \;\;\Leftrightarrow \;\;
\exists\, u_1,\ldots,u_m\in\bC:\,\sum_{x=1}^mu_x=0\;\;\mbox{and}\;\; \forall\, a,x: z_{ax}=u_x
\ee 
Observe that $J$ is a null-subspace of the operator system $l_{mn}^\infty$. Then $l_{mn}^\infty/J$ has the natural structure of operator system (see Section \ref{s:quot}).

\begin{proposition}
The operator system $U_{m,n}$ is unitally completely isomorphic to $l_{mn}^\infty/J$.
\end{proposition}
\begin{proof}
It is an immediate consequence of Proposition \ref{p:copr} (c.f. Remark \ref{r:free}).
\end{proof}
\begin{corollary}
\label{c:dod}
An element $t=\sum_{a,x}z_{ax}p_{ax}$ is positive if and only if there is a sequence $(w_{ax})\in J$ such that $z_{ax}+w_{ax}\geq 0$ for any pair $a,x$.
\end{corollary}


Define 
\begin{equation}
\label{e:Vmn}
V_{mn}=\left\{(z_{ax})_{a,x}\in l_{mn}^\infty:\, \sum_{a=0}^{n-1}z_{ax}\;\mbox{does not depend on $x$}\right\}
\end{equation}
The matrix order structure on $V_{m,n}$ is inherited from the described above order structure on $l_{mn}^\infty$. Namely
\begin{equation}
\label{e:Vmn2}
M_k(V_{m,n})^+=\{(A_{ax})_{a,x}:\,\mbox{$A_{ax}\geq 0$ and $\sum_{a}A_{ax}$ does not depend on $x$}\}
\end{equation}

Now we are ready to formulate a generalization of \cite[Proposition 5.11]{FKPT}
\begin{proposition}
\label{p:Ud}
The operator system dual $U_{m,n}^\mathrm{d}$ is completely order isomorphic to $V_{m,n}$.
\end{proposition}
\begin{proof}
The proof is the same as in \cite{FKPT}. Let $\gamma:l_{mn}^\infty\to U_{m,n}$ be defined by
\be \gamma((z_{ax}))=\frac{1}{n}\sum_{a,x}z_{ax}p_{ax}.\ee 
It follows from Remark \ref{r:free} and Proposition \ref{p:copr} that $\gamma$ is a complete quotient map. Thus according to \cite[Proposition 1.8]{FP} the adjoint map $\gamma^\mathrm{d}:U_{m,n}^\mathrm{d}\to (l_{mn}^\infty)^\mathrm{d}=l_{mn}^\infty$ is a complete order inclusion.
The map $\gamma^\mathrm{d}$ is given by
$\gamma^d(f)=(f(p_{ax}))_{a,x}$ for $f\in U_{m,n}^\mathrm{d}$
and its image is equal to $V_{m,n}$. This completes the proof.
\end{proof}

\section{Tensor products of $U_{m,n}$}
Let us remind that for any operator systems $V$ and $W$ and for any tensor product $\tau$ we have the following identities
\be M_k(V\otimes W)=M_k(V)\otimes W\ee 
and
\be M_k(V\otimes_\tau W)^+=(M_k(V)\otimes_\tau W)^+.\ee 
Moreover, we have
\begin{lemma}
For any operator system $V$ we have
\be M_k(V)^+=(M_k(\bC)\otimes_\mathrm{min} 
V)^+.
\ee 
\end{lemma}
\begin{proof}
We may assume that $V\subset B(H)$ is a concrete operator system.
Then the embedding $M_k(V)\subset M_k(B(H))$ defines the cone $M_k(V)^+$ of positive elements in $M_k(V)$. The statement of the proposition follows from the identification
$M_k(B(H))=M_k(\bC)\otimes B(H)$.
\end{proof}

Now, let us fix an operator system $V$. Our aim is to characterize positive elements in $M_k(U_{m,n}\otimes_\tau U_{m,n})$ for various tensor products $\tau$.  

Firstly, that any $t\in M_k(V\otimes U_{m,n})$ can written in the form
$t=\sum_{a,x}v_{ax}\otimes p_{ax}$
for some matrices $v_{ax}\in M_k(V)$. The matrices are not unique.
\begin{proposition}
\label{p:opi}
Let $v_{ax}\in M_k(V)$. Then
$\sum_{a,x}v_{ax}\otimes p_{ax}=0$ if and only if there are elements $v_x^0\in M_k(V)$ such that $v_{ax}=v_x^0$ for any $x$ and $\sum_xv_x^0=0$.
\end{proposition}
\begin{proof}
Similar to Proposition \ref{p:zero}.
\end{proof}
\begin{lemma}
\label{l:vax}
Let $k\in\bN$ and let $t= \sum_{a,x}v_{ax}\otimes p_{ax} $. If
$t\in M_k(V\otimes_\mathrm{min}U_{m,n})^+$ then 
$\sum_{a,x}z_{ax}v_{ax}\in M_k(V)^+$ for any $(z_{ax})\in V_{m,n}^+$. 
\end{lemma}
\begin{proof}
It follows from the definition of the $\mathrm{min}$ tensor product that 
$t$ is positive if and only if $\phi\otimes\psi(t)$ is positive for any unital completely positive maps
$\phi:V\to M_q(\bC)$ and $\psi:U_{m,n}\to M_r(\bC)$. Let $r=1$ and $\psi\in V_{m,n}^+$ be a functional defined as $\psi(p_{ax})=z_{ax}$ for any $a,x$, where $\psi=(z_{ax})$. Observe that
$\phi\otimes\psi(t)=\phi\left(\sum_{a,x}z_{ax}v_{ax}\right)$. Since $\phi$ is arbitrary, the positivity of $t$ is equivalent to positivity of the sum $\sum_{a,x}z_{ax}v_{ax}$ for any $(z_{ax})\in V_{m,n}^+$.
\end{proof}
As a consequence we get the following
\begin{proposition}
\label{p:sep}
Let $t= \sum_{a=1}^{m_1}\sum_{b=1}^{m_2}\sum_{x=1}^{n_1}\sum_{y=1}^{m_2}t_{abxy}p_{ax}\otimes p_{by}\in U_{m_1,n_1}\otimes U_{m_2,n_2}$. If $t\in (U_{m_1,n_1}\otimes_\mathrm{min}U_{m_2,n_2})^+$ then
\begin{equation}
\label{p:min}
\sum_{a,b,x,y}z_{ax}w_{by}t_{abxy}\geq 0
\end{equation}
for every $(z_{ax})\in V_{m_1,n_1}^+$ and $(w_{by})\in V_{m_2,n_2}^+$.
\end{proposition}
\begin{proof}
Let $v_{by}=\sum_{a,x}t_{abxy}p_{ax}$ for aby $b$ and $y$. By Lemma \ref{l:vax} $t$ is positive in minimal tensor product if and only if $\sum_{b,y}w_{by}v_{by}\in U_{m_1,n_1}^+$ for any $(w_{by})\in V_{m_2,n_2}^+$.
Observe that 
$\sum_{b,y}w_{by}v_{by}=\sum_{a,x}\left(\sum_{b,y}w_{by}t_{abxy}\right)p_{ax}.$
Hence, if $t$ is positive in minimal tensor product, $(w_{by})\in V_{m_2,n_2}^+$, and $(z_{ax})\in V_{m_1,n_1}^+$, then 
$\sum_{a,b,x,y}z_{ax}w_{by}t_{abxy}=\sum_{a,x}z_{ax}\left(\sum_{b,y}w_{by}t_{abxy}\right)\geq 0.$
\end{proof}

Let us notice, that the converse implication in the above proposition does not hold. It is a motivation to formulate the following 
\begin{definition}
Let $V$ and $W$ be ordered linear spaces. We say that a functional $\phi\in (V\otimes W)^\mathrm{d}$ is a \textit{separable positive functional} if $\varrho=\sum_{i=1}^m\varphi_i\otimes\psi_i$ for some $n\in\bN$, $\varphi_1,\ldots,\varphi_n\in V^+$ and $\psi_1,\ldots,\psi_n\in W$. The cone of all separable positive functional will be denoted by $(V^\mathrm{d})^+\otimes(W^\mathrm{d})^+$. We say that an element $t\in V\otimes W$ is \textit{block-positive}, if $\varrho(t)\geq 0$ for every $\varrho\in(V^\mathrm{d})^+\otimes(W^\mathrm{d})^+$. The cone of all block-positive elements in $V\otimes W$ will be denoted by $V\otimes_{\mathrm{bp}}W$.
\end{definition}

\section{Bipartite correlation boxes via operator systems $U_{m,n}$}
\label{sec:bipartite_boxes}

In order to study non-locality we consider the following scenario. Assume that there are two spatially separated and non communicating parties usually called Alice and Bob. They choose among $m$ different observables labeled by $x=1,\ldots,m$ for Alice and $y=1,\ldots,m$ for Bob. After measurment they emit some outcomes. Possible outcomes for Alice are labeled by $a=0,1,\ldots,n-1$ and for Bob by $b=0,1,\ldots,n-1$. Now, for any $x=1,\ldots,m$ and $a=0,\ldots,n-1$ we consider the probability $P(a|x)$ that Alice produces outcome $a$ provided that she was influenced by the input signal $x$. Analogously we define $P(b|y)$ for Bob where $y=1,\ldots,m$ and $b=0,\ldots,n-1$. Finally, if $x,y=1,\ldots,m$ and $a,b=0,1,\ldots,n-1$ let $P(ab|xy)$ be the probability that Alice and Bob produced the pair of outcomes $a$ and $b$ provided that they got inputs $x$ and $y$ respectively. The system $(P(ab|xy))_{a,b=0,1,\ldots,n-1,\;x,y=1,\ldots,m}$ will be called a correlation box.
The notion of correlation box was defined by Tsirelson \cite{T1,T2} in order to study Bell inequalities in quantum setting.

In this section we restate the Fritz's (\cite{F1}) characterization of various correlation boxes by different tensor products. We do it in terms of operator system tensor products of noncommuting cubes. This is the direct generalization  of \cite[Section 7]{FKPT}.

We consider the following classes of boxes

\begin{definition}
\label{d:box}
Given a box $(P(ab|xy))$ we say that it 
\begin{enumerate}
\item
is \textit{non-signalling} if
\begin{equation}
\label{e:ns1}
\sum_{a=0}^{n-1}P(ab|xy)=P(b|y)\;\mbox{for any $x=1,\ldots,m$},
\end{equation}
\begin{equation}
\label{e:ns2}
\sum_{b=0}^{n-1}P(ab|xy)=P(a|x)\;\mbox{for any $y=1,\ldots,m$}.
\end{equation}
\item
has \textit{local hidden variable (LHV)} if there are two families $P_\lambda(a|x)$ and $Q_\lambda(b|y)$ of probability distributions (it means $P_\lambda(a|x)\geq 0$ and $\sum_aP_\lambda(a|x)=1$, $x=1,\ldots,m$, and similarly for $Q_\lambda(b|y)$) and positive numbers $r_\lambda$ with $\sum_\lambda r_\lambda=1$ such that
\be P(ab|xy)=\sum_\lambda P_\lambda(a|x) Q_\lambda(b|y).\ee 
\item
is \textit{quantum} if there are 
\begin{enumerate}
\item
two Hilbert spaces $H_A$ and $H_B$, 
\item
two sets of positive operators $(E_x^a)_{x,a}$ and $(F_y^b)_{y,b}$ acting on $H_A$ and $H_B$ respectively with the property 
$$\sum_aE_x^a=\jed_A\;\mbox{for any $x=1,\ldots,m$},$$
$$\sum_bF_y^b=\jed_B\;\mbox{for any $y=1,\ldots,m$},$$
\item
a density matrix $\varrho$ on $H_A\otimes H_B$
\end{enumerate}
such that
$$P(ab|xy)=\Tr(\varrho(E_x^a\otimes F_y^b))$$
for any $x,y=1,\ldots,m$ and $a,b=0,1,\ldots,n-1$.
\item
is \textit{contextual} if there are 
\begin{enumerate}
\item
a Hilbert space $H$, 
\item
two sets of positive operators $(E_x^a)_{x,a}$ and $(F_y^b)_{y,b}$ acting on $H$ with the property 
$$\sum_aE_x^a=\jed\;\mbox{for any $x=1,\ldots,m$},$$
$$\sum_bF_y^b=\jed\;\mbox{for any $y=1,\ldots,m$},$$
and
$$E_x^aF_y^b=F_y^bE_x^a\;\mbox{for any $x,y,a,b$}$$
\item
a density matrix $\varrho$ on $H$
\end{enumerate}
such that
$$P(ab|xy)=\Tr(\varrho E_x^a F_y^b)$$
for any $x,y=1,\ldots,m$ and $a,b=0,1,\ldots,n-1$.
\end{enumerate}
\end{definition}
Let $m,n\in\bN$ be fixed. We let $\cL$ denote the class of LHV boxes, $\cQ$ -- the class of quantum boxes, $\cC$ -- the class of contextual boxes and $\cP$ -- the class of non-signalling boxes. Then it is known that $\cL\subsetneq\cQ\subset\cC\subsetneq\cP$.
In \cite[Theorem 7.1]{FKPT} it was show that if $m=n=2$ then various classes of correlation boxes can be characterized by means of different tensor products of non-commuting cubes. Now, having "generalized cubes" $U_{m,n}$ (see Remark \ref{r:cub}) we can formulate the theorem for arbitrary numbers inputs and outputs.

For an operator system $V$ let $S(V)$ denote its \textit{state space}, i.e.
\be S(V)=\{\phi\in V^\mathrm{d}:\,\phi(V^+)\subset[0,\infty),\;\phi(e)=1\}.\ee 
\begin{theorem}
\label{thm:main}
We have the following equalities
\begin{enumerate}
\item
$\cP=\{(\phi(p_{ax}\otimes p_{by})):\,\phi\in S(U_{m,n}\otimes_\mathrm{max}U_{m,n})\}$
\item
$\cC=\{(\phi(p_{ax}\otimes p_{by})):\,\phi\in S(U_{m,n}\otimes_\mathrm{c}U_{m,n})\}$
\item
$\cQ=\{(\phi(p_{ax}\otimes p_{by})):\,\phi\in S(U_{m,n}\otimes_\mathrm{min}U_{m,n})\}$
\end{enumerate}
\end{theorem}
\begin{proof}
The idea of the proof is basically the same as in \cite{FKPT}.

(1) Let $\cP'=\{\phi(p_{ax}\otimes p_{by}):\,\phi\in S(U_{m,n}\otimes_\mathrm{max}U_{m,n})\}$. Then obviously $\cP'\subset\cP$. To see the converse inclusion notice that any correlation box $P=(P(ab|xy))$ is an element of $(l_{mn}^\infty\otimes_\mathrm{min} l_{mn}^\infty)^+$. This is the consequence of positivity of numbers $P(ab|xy)$. Moreover, let us observe that if $P\in\cP$ then due to conditions (\ref{e:ns1}) and (\ref{e:ns2}) $P\in V_{m,n}\otimes V_{m,n}$ (c.f. (\ref{e:Vmn})). Hence $P\in(V_{m,n}\otimes_\mathrm{min}V_{m,n})^+$. Minimal and maximal tensor products are dual to each other. Thus it follows from Proposition \ref{p:Ud} that $(P(ab|xy))\in (V_{m,n}\otimes_\mathrm{min}V_{m,n})^+$ if and only if $P(ab|xy)=\phi(p_{ax}\otimes p_{by})$ for some positive functional $\phi:U_{m,n}\otimes_\mathrm{max}U_{m,n}\to\bC$. From (\ref{e:ns1}) and (\ref{e:ns2}) we get normalization $\phi(\jed\otimes\jed)=1$. Hence $\cP\in\cP'$

(2)
As previously, let $\cC'=\{\phi(p_{ax}\otimes p_{by}):\,\phi\in S(U_{m,n}\otimes_\mathrm{c}U_{m,n})\}$. Assume that $P\in\cC$. Then 
\be P(ab|xy)=\Tr(\varrho E_x^a F_y^b)\ee  for some $E_x^a$ and $F_y^b$ acting on a Hilbert space $H$ such that $E_x^a\geq 0$, $F_y^b\geq 0$, $\sum_aE_x^a=\jed_A$, $\sum_bF_y^b=\jed_B$ and $E_x^aF_y^b=F_y^bE_x^a$. Hence the linear maps $\alpha,\beta:U_{m,n}\to B(H)$ defined by 
$\alpha(p_{ax})=E_x^a$ for any pair $x,a$ and $\beta(p_{by})=F_y^b$ for any pair $y,b$ are unital completely positive and have commuting ranges. 
It follows from the definition of commuting tensor product (\cite[Section 6]{tens}) that the map $\alpha\cdot\beta:U_{m,n}\otimes_\mathrm{c}U_{m,n}\to B(H)$ defined by $(\alpha\cdot\beta)(u\otimes w)=\alpha(u)\beta(w)$, $u,w\in U_{m,m}$, is a completely positive map. Thus the linear functional $\phi: U_{m,n}\otimes_\mathrm{c}U_{m,n}\to\bC$ by $\phi(t)=\Tr(\varrho(\alpha\cdot\beta)(t))$ turns out to be a state on $U_{m,n}\otimes_\mathrm{c}U_{m,n}$ such that $P(ab|xy)=\phi(E_x^a\otimes F_y^b)$. Hence $P\in\cC'$ and therefore $\cC\subset\cC'$.

Conversely, assume $P\in\cC'$, hence $P(ab|xy)=\phi(E_x^a\otimes F_y^b)$ for some state $\phi$ on $U_{m,n}\otimes_\mathrm{c}U_{m,n}$. By \cite[Corollary 6.5]{tens} there is a Hilbert space $H$, two $^*$-homomorphism $\pi_A,\pi_B:C^*(\bZ_n^{*m})\to B(K)$ with commuting ranges and a unit vector $\xi\in K$ such that $\phi(u\otimes w)=\langle\xi,\pi_A(u)\pi_B(w)\xi\rangle$ for any $u,w\in U_{m,n}$. Now, define $E_x^a=\pi_A(p_{ax})$, $F_y^b=\pi_B(p_{by})$ and $\varrho=|\xi\rangle\langle\xi|$. Then $P(ab|xy)=\Tr(\varrho E_x^aF_y^b)$ and consequently $P\in\cC$. Hence $\cC'\subset\cC$.

(3)
Let $\cQ'=\{\phi(p_{ax}\otimes p_{by}):\,\phi\in S(U_{m,n}\otimes_\mathrm{min}U_{m,n})\}$. Assume that $P\in \cQ$. Then 
\be P(ab|xy)=\Tr(\varrho(E_x^a\otimes F_y^b))\ee  for some $E_x^a$ and $F_y^b$ acting on $H_A$ and $H_B$ respectively such that $E_x^a\geq 0$, $F_y^b\geq 0$, $\sum_aE_x^a=\jed_A$ and $\sum_bF_y^b=\jed_B$. Due to the conditions one can consider maps $\alpha:U_{m,n}\to B(H_A)$ and $\beta:U_{m,n}\to B(H_B)$ defined by
$\alpha(p_{ax})=E_x^a$ for any pair $x,a$ and $\beta(p_{by})=F_y^b$ for any pair $y,b$. The maps $\alpha$ and $\beta$ are unital completely positive. Then it follows from \cite[Theorems 4.4 and 4.6]{tens} that the map
\be \alpha\otimes\beta:U_{m,n}\otimes_\mathrm{min}U_{m,n}\to B(H_A\otimes H_B)\ee  is again unital and completely positive. Hence the linear functional \be \phi:U_{m,n}\otimes_\mathrm{min}U_{m,n}\to\bC\ee  
defined by $\phi(t)=\Tr(\varrho(\alpha\otimes\beta)(t))$ is a state on $U_{m,n}\otimes_\mathrm{min}U_{m,n}$ and $P(ab|xy)=\phi(p_{ax}\otimes p_{by})$. Therefore $\cQ\subset\cQ'$.

Now, assume $P\in\cQ'$. Thus $P(ab|xy)=\phi(p_{ax}\otimes p_{by})$ for some state $\phi$ on $U_{m,n}\otimes_\mathrm{min}U_{m,n}$. Let $H$ be a Hilbert space such that $C^*(\bZ_n^{*m})\subset B(H)$. Then $C^*(\bZ_n^{*m})\otimes_{\mathrm{min}}C^*(\bZ_n^{*m})\subset B(H\otimes H)$ and $\phi$ can be extended to a state $\tilde{\phi}$ on $C^*(\bZ_n^{*m})\otimes_{\mathrm{min}}C^*(\bZ_n^{*m})$. Due separability of this algebra we may assume $\tilde{\phi}$ is the restriction of a normal state on $B(H\otimes H)$ to $C^*(\bZ_n^{*m})\otimes_{\mathrm{min}}C^*(\bZ_n^{*m})$. Therefore, there exists a positive and trace class operator $\varrho$ on $H\otimes H$ such that ${\phi}(t)=\Tr(\varrho t)$ for $t\in U_{m,n}\otimes_\mathrm{min}U_{m,n}$. Now, let us define $H_A=H_B=H$. If $\iota:U_{m,n}\to C^*(\bZ_n^{*m})$ is the canonical embedding then let
$E_x^a=\iota(p_{ax})$ and $F_y^b=\iota(p_{by})$. Thus $P(ab|xy)=\Tr(\varrho(E_x^a\otimes F_y^b))$, and hence $\cQ'\subset\cQ$.
\end{proof}


\section{A nutshell of NPA hierarchy}

\subsection{Quantum behaviors}

NPA hierarchy is an infinite hierarchy of conditions necessarily satisfied by any set of quantum correlations \cite{NPA2008}.



\begin{definition}\label{defi:qbehavior}
The behavior $P$ is a quantum behavior if there exists a pure (normalized) state $|\psi\rangle$ in a Hilbert space $H$, a set of measurement operators
$\{E_a: a\in \widetilde{A}\}$ for Alice, and a set of measurement operators $\{E_b: b\in \widetilde{B}\}$ for Bob such that for all $a\in \widetilde{A}$ and 
$b\in \widetilde{B}$
\ben
&&P(a) = \langle\psi |E_a| \psi \rangle,\nonumber \\
&&P(b) = \langle\psi |E_b| \psi \rangle,\nonumber \\
&&P(a,b) = \langle\psi |E_aE_b| \psi \rangle,
\een
with the measurement operators satisfying

\begin{enumerate}

\item $E_a^\dagger = E_a$ and $E_b^\dagger = E_b,$

\item $E_a E_{a'}= \delta_{aa'} E_a$ if $X(a)=X(a')$ and $E_b E_{b'}= \delta_{bb'} E_b$ if $Y(b)=Y(b'),$

\item $[E_a,\; E_b]=0.$
\end{enumerate}
\end{definition}

\subsection{Sets of operators and sequences}
Let $\widetilde{\mathcal{E}}$ denote the set of projectors of Definition \ref{defi:qbehavior} plus the identity, i.e. 
$\widetilde{\mathcal{E}}= \un \cup \{E_a: a\in \widetilde{A}\} \cup \{E_b: b\in \widetilde{B}\}.$

Let $\mathcal{O} = \{O_1,\cdots, O_n\}$ be a set of $n$ operators, where each $O_i$ is a linear combination of products of projectors in $\widetilde{\mathcal{E}}.$ Define $\mathcal{F}(\mathcal{O})$ as the set of all independent equalities of the form

\begin{equation}\label{eq:O_i}
\sum_{i,j} (F_k)_{i,j} \langle\psi|O_i^{\dagger} O_j| \psi\rangle = g_k(P) \;\;\; k=1,\cdots,m 
\end{equation}

which are satisfied by the operators $O_i,$ where the coefficients $g_k(P)$ are linear functions of the probabilities $P(a,b):$

\begin{equation}\label{eq:g_k}
g_k(P) = (g_k)_0 + \sum_{a,b} (g_k)_{ab} P(a,b).
\end{equation}

Let a sequence $S$ be a product of projectors in $\widetilde{\mathcal{E}}.$ The length $|S|$ of a sequence is the minimum number of projectors needed to generate it. We define $S_n$ to be the set of sequences of length smaller than or equal to $n$ (excluding null sequences). 

\begin{equation}
\begin{split}
S_0 &= \{ \un\}\\
S_1 &= S_0 \cup \{E_a: a\in \widetilde{A}\} \cup \{E_b: b\in \widetilde{B}\}\\
S_2 &= S_0 \cup S_1 \cup \{E_aE_{a'}: a, a' \in \widetilde{A}\} \cup \{E_bE_{b'}: b, b' \in \widetilde{B}\} \cup \{E_aE_b: a\in \widetilde{A}, b\in \widetilde{B}\}\\
S_3 &= \cdots
\end{split}
\end{equation}
Any operator $O_i \in \mathcal{O}$ can be written as a linear combination of operators in $S_n$ for $n$ sufficiently large.

\subsection{A hierarchy of necessary and sufficient conditions}

A certificate $\Gamma^n$ associated to the set of operators $S_n$ is a real positive semi-definite matrix with entries $\{ \Gamma_{s,t}^n: |s|,|t|\leq n,\}$ that satisfies the linear equalities

\begin{equation}\label{eq:certificate1}
\Gamma^n_{1,1} =1, \; \Gamma^n_{1,a} = P(a), \; \Gamma_{1,b}^n = P(b), \; \Gamma_{a,b}^n = P(a,b)
\end{equation}
for all $a\in \widetilde{A}$ and $b\in \widetilde{B},$ and 

\begin{equation}\label{eq:certificate2}
\Gamma^n_{s,t} = \Gamma^n_{u,v} = P(a), \;\; if \;\; S^{\dagger} T= U^{\dagger} V
\end{equation}
for all $|s|,|t|,|u|,|v| \leq n.$ Where the index $s$ associated with the sequence operators $S.$

From Proposition 4 of \cite{NPA2008}, we have $\Gamma_{s,t}^n = \langle \psi|S^\dagger T |\psi \rangle$ if $P$ is a quantum behavior. On the other hand, by the Theorem 8 of \cite{NPA2008}, for a behavior $P,$ the existence of certificate $\Gamma^n$ for all $n\geq 1$ is sufficient to deduce that $P$ is a quantum behavior. 

\subsection{NPA hierarchy vs noncommutative cube}

To unify such two categories we assume Alice (resp. Bob) will choose $X=1,\cdots,m$ (resp. $Y=1,\cdots,m$) inputs and each input will have $a=0,\cdots,n-1$ (resp. $b=0,\cdots,n-1$) outputs.

\begin{theorem}
\label{t:npa}
Following two statements are equivalent:

\begin{enumerate}

\item $P$ is a behavior such that there exists a certificate $\Gamma^n$ of order $n$ for all $n\geq 1.$

\item There exists $\varphi \in S(U_{m,n} \otimes_c U_{m,n}),$ such that $P= \{(\varphi(P_{aX} \otimes P_{bY}))\}.$

\end{enumerate}
\end{theorem}

This theorem of course follows from the result of \cite{NPA2008} just mentioned  in previous section,
and our theorem \ref{thm:main}. However  here we will give a direct proof which has its independently interest. 
Before we give the proof, we need the following easy observation:


\begin{lemma}

\begin{enumerate}
\item[(i)] The set $\bigcup_n S_n$ is isomorphism to the group $\mathbb{Z}_n^{*m} \times  \mathbb{Z}_n^{*m},$ i.e. every sequence $S\in \bigcup_nS_n$ is one to one correspond to a reduced ward $s\in \mathbb{Z}_n^{*m} \times \mathbb{Z}_n^{*m}.$
 
\item[(ii)] If $\A$ is the *-algebra generated by $\bigcup_n S_n,$ then $\A$ is *-isomorphism to the group algebra $\mathbb{C}(\mathbb{Z}_n^{*m} \times \mathbb{Z}_n^{*m}).$
\end{enumerate}   
\end{lemma}

\begin{proof}
(i). Let $s_X$ denote the generator of the $X$-th copy. Define the map $\pi: \bigcup_n S_n \mapsto  \mathbb{Z}_n^{*m} \times \mathbb{Z}_n^{*m}$ as:
\begin{equation}
\begin{split}
 \pi(E_a) &= (s_X^{a+1}, \un), \;\; a \in X(a), \\ \pi(E_b) &= (\un, s_Y^b), \;\; b\in Y(b),\\
 \pi(E_aE_{a'}) &= (s_X^{a+1} s_{X'}^{a'+1},\un), \;\; a \in X(a), a'\in X'(a'),\\
 \pi(E_bE_{b'}) &= (\un ,s_Y^{b+1} s_{Y'}^{b'+1}), \;\; b \in Y(b), b'\in Y'(b'),\\
 \pi(E_aE_{b}) &= (s_X^{a+1} ,s_Y^{b+1}), \;\;  a\in X(a), b \in Y(b),\\
 &\cdots
\end{split}
\end{equation}
It is easy to check that $\pi$ is a bijection.

(ii). Define the map $\tau: \A \mapsto \mathbb{C}(\mathbb{Z}_n^{*m} \times \mathbb{Z}_n^{*m})$ as:
\begin{equation}
\begin{split}
 \tau(E_a) &= \Big(\frac{1}{n} \sum_{j=0}^{n-1} \omega^{ja}\delta_{s_X}^j, \delta_{\un} \Big), \;\; a\in X(a), \\ \tau(E_b) &= \Big(\delta_{\un}, \frac{1}{n} \sum_{j=0}^{n-1} \omega^{jb}\delta_{s_Y}^j \Big), \;\; b\in Y(b),\\
 \tau(E_aE_{a'}) &= \Big( \frac{1}{n^2} \sum_{j,j'=0}^{n-1} \omega^{ja+j'a'}\delta_{s_X}^j\delta_{s_{X'}}^{j'}, \delta_{\un} \Big), \;\; a \in X(a), a'\in X'(a'),\\
 \tau(E_bE_{b'}) &= \Big( \delta_{\un} , \frac{1}{n^2} \sum_{j,j'=0}^{n-1} \omega^{jb+j'b'}\delta_{s_Y}^j\delta_{s_{Y'}}^{j'} \Big), \;\; b \in Y(b), b'\in Y'(b'),\\
 \tau(E_aE_{b}) &= \Big( \frac{1}{n} \sum_{j=0}^{n-1} \omega^{ja}\delta_{s_X}^j ,\frac{1}{n} \sum_{j=0}^{n-1} \omega^{jb}\delta_{s_Y}^j\Big), \;\;  a\in X(a), b \in Y(b),\\
 &\cdots
\end{split}
\end{equation}
where the $\{(\delta_{s_X}, \delta_{\un}), (\delta_{\un}, \delta_{s_Y}), (\delta_{s_X}, \delta_{s_Y}): X=1,\ldots, m, Y=1,\ldots,m\}$ is the basis of $\mathbb{C}( \mathbb{Z}_n^{*m} \times \mathbb{Z}_n^{*m}).$ It is clear that $\tau$ is a *-isomorphism.
\end{proof}

\begin{remark}\label{remark:O}
By (ii) of this lemma, every operator $O_i \in \mathcal{O}$ is an element of $\mathbb{C}( \mathbb{Z}_n^{*m} \times \mathbb{Z}_n^{*m}).$
\end{remark}

\begin{proof}[Proof of Theorem \ref{t:npa}] In this proof, we denote group $\mathbb{Z}_n^{*m} \times \mathbb{Z}_n^{*m}$ by $G.$ We first prove (i)$\Rightarrow$(ii). 
As we know from the proof of Theorem 8 in \cite{NPA2008}, there exists a semi-positive matrix $\Gamma^{\infty}= (\Gamma^{\infty}_{s,t}),$ where index $s,t$ associate with the sequence $S, T \in \bigcup_n S_n.$ By the previous Lemma. There also exists a semi-positive matrix, we still denote as $\Gamma= (\Gamma^{\infty}_{s,t}), s, t\in G.$ Because of its positivity, there exists an infinite family of vectors $\{|s\rangle \in \ell_2(G), s\in G\}$ such that 
$\Gamma_{s,t} = \langle s|t\rangle.$ Now let $\lambda: G \mapsto B(\ell_2(G))$ be the left regular unitary representation of $G.$ Then we can get an induced *-algebra representation $\pi: \mathbb{G} \mapsto B(\ell_2(G))$ \cite{F2},

\be \pi\left(\sum_g x_g \delta_g\right) = \sum_g x_g \lambda(g), x_g \in \mathbb{C}.\ee   

Now define a functional $\varphi: \mathbb{C}(G) \mapsto \mathbb{C}$ as following:
\be \varphi(x) = \langle \un| \pi(x) \un \rangle.\ee 
Since $|\varphi(x)| \leq \|\pi(x)\|_{B(\ell_2(G))} \leq \|x\|_{C^*(G)}$ for any $x\in \mathbb{C}(G),$ thus by Hahn-Banach theorem, $\varphi$ can extend to a functional on $C^*(G)$ with norm less or equal one (actually the norm of $\varphi$ is one). The positivity and $\varphi(P_{aX} \otimes P_{bY})= P(a,b), a\in X(a), b\in Y(b)$ is clear. Now we already get a state $\varphi$ on $C^*(G)$ such that $P= (\varphi(P_{aX} \otimes P_{bY})).$ Since $C^*(G) = C^*(*_m\mathbb{Z}_n) \otimes_{\max} C^*(*_m\mathbb{Z}_n)$ and
$U_{m,n} \otimes_c U_{m,n} \subseteq_{coi} C^*(*_m\mathbb{Z}_n) \otimes_{\max} C^*(*_m\mathbb{Z}_n),$ we complete the proof. 

$(ii) \Rightarrow (i).$ Suppose we have a state $\varphi$ on $U_{m,n} \otimes_c U_{m,n},$ i.e. on $C^*(G).$ Then by the GNS construction, there exist a Hilbert space $H,$ a cyclic vector $\xi \in H$ and a *-representation $\pi: C^*(G) \mapsto B(H),$ such that

\be \varphi(x) = \langle \xi| \pi(x) \xi \rangle,\;\; x\in C^*(G).\ee 
By the remark \ref{remark:O}, we have $O_i \in \mathbb{C}(G),$ thus we can define a matrix $\Gamma_{i,j}$ as following:
\be \Gamma_{i,j} = \langle \xi| \pi(O^\dagger_i O_j) \xi \rangle.\ee  To prove it is a certificate associated to the set of operator $S_n$ is similar to the proof of Proposition 4 in \cite{NPA2008}.
\end{proof}
\begin{remark}
It was mentioned by Fritz in \cite[Remark 3.5]{F2} that NPA hierarchies can be described by states on some suitable tensor products for group
C*-algebras. Theorem \ref{t:npa} basically restates this remark but in the context of operator systems. 
\end{remark}
\section{Approximation of $U_{2,2}\otimes U_{2,2}$}
The Hilbert space and projectors that can reproduce any quantum box is quite huge, in particular, nonseparable.
There is a question whether in some cases we can have finite-dimensional  Hilbert space which will do the job, or at least,
whether one can approximate the set of quantum boxes by boxes coming from finite-dimensional Hilbert spaces. 
Here we shall provide elementary construction of such approximation in the case of $n=m=2$. 
In such case, all needed projectors can be constructed out of two, which we will  choose in such a way, that 
the principal angles will be placed uniformly on a quarter of circle. 
We shall use result of 
Masanes \cite{Masanes} who showed  that for $m=n=2$ any point of $\cQ$ 
can be realized on Hilbert space  $\bC^2\otimes \bC^2$, i.e. for any box $p(ab|xy)\in\cQ$ 
there exist state $\psi\in \bC^2 \otimes \bC^2$ and projectors $P_{ax}^A$, $P_{by}^B$ 
with $\sum_aP_{ax}^A=I_A$ and $\sum_b P_{by}^B=I_B$ such that 
\be
p(ab|xy)=\<\psi| P_{ax}^A\ot P_{by}^B |\psi\>.
\label{eq:ext}
\ee
We now want to construct set of projectors on a larger Hilbert space, that would universally work for all boxes 
(the dimension would depend on the needed accuracy of approximation). 
To this end we consider  Hilbert space
\be
\label{eq:hilbert-space}
H=H_A\ot H_B, \quad H_A=\bigoplus_{k=1}^{N} H_k^A,\quad 
H_B=\bigoplus_{k=1}^{N} H_k^B
\ee
with $H^A_{k}\simeq H^B_l\simeq\bC^2$.
Now, Alice and Bob will have the same projectors given by 

\begin{eqnarray}
\label{eq:kl}
&&\Pr_{0,0}=\oplus_{k=1}^N |0\>\<0|,\quad \Pr_{1,0}=\un-\Pr_{0,0}, \nonumber \\
&&\Pr_{0,1}=\oplus_{k=1}^N |\psi_k\>\<\psi_k| ,\quad \Pr_{1,1}=\un-\Pr_{0,1}
\end{eqnarray}
where 
\be
\label{eq:psi_k}
\psi_k=\cos \alpha_k |0\> + \sin \alpha_k |1\>
\ee
with $\alpha_k=\frac{k\pi}{2N}$.
Thus $\Pr_{0,0}$ and $\Pr_{1,0}$ are chosen in such a way, that the principal angles (which are $\alpha_k$ in this case) fill uniformly 
the quarter of circle.  

Let us  define projectors acting on $\bC^2$
\begin{eqnarray}
&&P_{0,0}^{Ak}=P_{0,0}^{Bk}=|0\>\<0|,\quad P_{1,0}^{Ak}=P_{1,0}^{Bk}=|1\>\<1|\nonumber\\
&&P_{0,1}^{Ak}=P_{0,1}^{Bk}=|\phi_k\>\<\phi_k|,\quad P_{1,1}^{Ak}=P_{1,1}^{Bk}=\un - |\phi_k\>\<\phi_k|.
\end{eqnarray}

Now, for any state $\rho$ define $p_\rho$ as 
\be
p_\rho(ab|xy)=\Tr(\rho \Pr_{ax}\ot \Pr_{by}).
\ee
Further let us consider the following norm between two boxes by 
\be
||p-p'||_1=\sum_{a,b,x,y} |p(ab|xy) - p'(ab|xy)|.
\ee
We now prove the following proposition, showing, that our projectors \eqref{eq:kl} together with arbitrary state, 
can be used to approximate arbitrary quantum box:
\begin{proposition} 
For arbitrary $\epsilon>0$ and a box $p$ with $m=n=2$ from $\cQ$ we can find $N$ and a  state $\rho$ acting on 
the Hilbert space \eqref{eq:hilbert-space} such that 
\be
||p-p_\rho||\leq \epsilon
\ee
\end{proposition}

\begin{proof}
Since $p\in \cQ$, and $n=m=2$ we can apply Masanes construction so that $p$ is given by \eqref{eq:ext}. 
Applying suitable unitary $U_A\ot U_B$ to the state and projectors in \eqref{eq:ext} we obtain 
\be
p(ab|xy)=\<\psi'|Q_{ax}^A\ot Q_{by}^B |\psi'\>
\ee
where  $\psi'=U_A\ot\U_B\psi$ and
\begin{eqnarray}
&&Q_{00}^A=Q_{00}^B=|0\>\<0|,\quad Q_{01}^A=|\psi\>\<\psi|, \quad Q_{01}^B=|\phi\>\<\phi|\nonumber \\
&&Q_{10}^A=Q_{10}^B=|1\>\<1|,\quad Q_{01}^A=|\psi^\perp\>\<\psi^\perp|, \quad Q_{01}^B=|\phi^\perp\>\<\phi^\perp|\nonumber \\
\end{eqnarray}
with
\be
\psi=\cos \alpha|0\>+ \sin \alpha |1\>, \quad \phi=\cos \beta|0\>+ \sin \beta |1\>
\ee
with $\alpha,\beta\in[0,\pi/2]$. We now find $k_0$ and $l_0$ such that 
\be
\label{eq:scalar}
\<\psi_{k_0}|\psi\>\leq \cos\frac{\pi}{4N},\quad \<\psi_{l_0}|\phi\>\leq \cos\frac{\pi}{4N}
\ee
where $\psi_k$ are given by \eqref{eq:psi_k}.
We then treat the two qubit state $|\psi_{k_0}\>|\psi_{l_0}\>$ as acting on Hilbert space $H_{k_0}^A\ot H^B_{l_0}$,
and let $\psi_{k_0,l_0}$ to be the above state embedded into the total Hilbert space $H_{AB}$ of 
\eqref{eq:hilbert-space}. We then have  for $\rho=|\psi_{k_0,l_0}\>\<\psi_{k_0,l_0}|$ 
\be
p_\rho(ab|xy)=
\<\psi_{k_0,l_0} | \Pr_{ax}\ot \Pr_{by} |\psi_{k_0,l_0}\>  =
\<\psi_{k_0,l_0} | P_{ax}^{Ak_0}\ot P_{by}^{Bl_0} |\psi_{k_0,l_0}\>.
\ee
One then finds that 
\begin{eqnarray}
&&\sum_{a,b,x,y} |p_\rho(ab|xy) - p(ab|xy)|\nonumber\\
&&= \sum_{a,b,x,y} |\<\psi_{k_0,l_0} | p_{ax}^{Ak_0}\ot p_{by}^{Bl_0}| \psi_{k_0,l_0}\> -  
\<\psi'|Q_{ax}^A\ot Q_{by}^B |\psi'\>| \nonumber\\
&&\leq 8\sqrt{1-\<\psi_{k_0}|\psi'\>} +8\sqrt{1-\<\psi_{l_0}|\psi'\>}
\leq 16 \sin^2\biggl(\frac{\pi}{4}\biggr) \leq \frac{\pi^2}{N^2}
\end{eqnarray}
where we used \eqref{eq:scalar}, and $\psi_i$ are given by \eqref{eq:psi_k}.
Thus taking $N\geq \frac{\pi}{\sqrt\ep}$, we obtain 
$\sum_{a,b,x,y} |p_\rho(ab|xy) - p(ab|xy)|$ which ends the proof.
\end{proof}

The construction seem not be possible if there are more than two inputs on both sides.
Namely, it is quite likely, that the value of so called $I_{2233}$ Bell's inequality does not attains maxima on any 
finite Hilbert space \cite{VidickWehner}. If true that would imply, that there exist correlation boxes that cannot be represented on 
finite dimensional Hilbert space.  Let us also mention, that \cite{WernerScholz} 
it is proved that if Tsirelson problem has positive solution for some $m,n$, 
then for any quantum box, there is finite-dimensional approximation  of the box, 
where the projectors depend on the box.
This allows  to build  $\epsilon$-approximation of any box with fixed projectors, that do not depend on the box 
in an analogous way as above, if only Tsirelson problem has positive answer.
For $n=m=2$ we know more - namely any box is exactly reproduced by use of two qubits, hence the construction 
above is so simple.

\section{Steering vs operator systems $U_{m,n}$}
Now, assume that Alice 
can choose between $m$ measurement settings, each of which can result in one of $n$ outcomes. Suppose Bob has $d$-dimensional quantum system. Following \cite{Pus} we define an assemblage as a set $(\sigma(a|x))_{x=1,\ldots,m,\,a=0,\ldots,n-1}$ such that $\sigma(a|x)\in M_d(\bC)^+$ for any pair $x,a$, the sum $\sum_a\sigma(a|x)$ does not depend on $x$ and $\Tr\left(\sum_a\sigma(a|x)\right)=1$ for every $x$.

\begin{definition}
We will say that an assemblage $(\sigma(a|x))$ 
\begin{enumerate}
\item
has \textit{local hidden state (LHS)} when 
\be \sigma(a|x)=\sum_\lambda p_\lambda(a|x)\sigma_\lambda\ee 
for some positive numbers $r_\lambda$ with $\sum_\lambda r_\lambda=1$, a family $P_\lambda(a|x)$ of probability distributions (i.e. $p_\lambda(a|x)\geq 0$ and $\sum_aP_\lambda(a|x)=1$ for each $\lambda$) and a family of positive matrices $\sigma_\lambda\in M_d(\bC)^+$ such that $\Tr\sigma_\lambda=1$.
\item
is \textit{quantum} if
\be 
\sigma(a|x)=\Tr_A(\varrho(E_x^a\otimes\un_B))
\ee 
where $E_x^a$ are some positive operators acting on a Hilbert space $H_A$ such that $\sum_aE_x^a=\jed_A$ for every $x$, and $\varrho$ is a density matrix on $H_A\otimes\bC^d$.
\end{enumerate}
\end{definition}
Let natural numbers $m,n,d$ be fixed. We let $\cP_\mathrm{s}$ denote the class of all assemblages, $\cQ_\mathrm{s}$ -- the class of all quantum assemblages and $\cL_\mathrm{s}$ -- the class of LHS assemblages. Obviously
$\cL_\mathrm{s}\subset\cQ_\mathrm{s}\subset\cP_\mathrm{s}$. It was noticed by Schr{\"o}dinger that $\cP_\mathrm{s}=\cQ_\mathrm{s}$ (see \cite{HJW} for a discussion on that topic).

For an operator system $V$ and a Hilbert space $H$ let $CP(V,H)$ denote the set of all completely postive maps $\alpha:V\to S_1(H)$ such that $\Tr\,\alpha(\jed)=1$. 

If $\xi\in\bC^d$ is a unit vector then it can considered as an element of $M_{1,d}(\bC)$. So for any $v\in V^+$ we have that $\xi^*v\xi\in M_d(V)^+$. But $\xi^*v\xi$ is nothing but $|\xi\rangle\langle\xi|\otimes v$. Let $M_d(\bC)^+\otimes V^+$ denote the subcone of $M_d(V)^+$ generated by all elements of that form. We will refere to elements of $M_d(\bC)^+\otimes V^+$ as separable elements.
\begin{theorem}
We have the following equality
\begin{equation}
\cQ_\mathrm{s}=\{(\alpha(p_{ax})):\,\alpha\in CP(U_{m,n},\bC^d)\},
\end{equation}
\end{theorem}
\begin{proof}
%
Let $\cQ_\mathrm{s}'=\{(\alpha(p_{ax})):\,\alpha\in CP(U_{m,n},\bC^d)\}$.
Assume $\sigma\in\cQ_\mathrm{s}$. So 
\be \sigma(x|a)=\Tr_A(\varrho(E_x^a\otimes\un_d)),\ee  
where $E_x^a$ act on a Hilbert space $H_A$, $E_x^a\geq 0$, $\sum_aE_x^a=\un_A$ and $\varrho$ is a density matrix on $H_A\otimes\bC^d$. Let $\varphi:U_{m,n}\to B(H_A)$ be a unital completely positive map such that $\varphi(p_{ax})=E_x^a$. Now, let 
$\alpha:U_{m,n}\to M_d(\bC)$ be defined by
\be \alpha(s)=\Tr_A\left(\varrho^\frac{1}{2}(\varphi(s)\otimes\un_d)\rho^{\frac{1}{2}}\right),\qquad s\in U_{m,n}.\ee 
Clearly, $\alpha$ is completely positive and $\sigma(a|x)=\alpha(p_{ax})$ for every $a$ and $x$. Therefore, $\sigma\in\cQ_\mathrm{s}'$.

Conversely, assume that $\sigma\in\cQ_\mathrm{s}'$, i.e. $\sigma(a|x)=\alpha(p_{ax})$ for some completely positive $\alpha:U_{m,n}\to M_d(\bC)$. Let $H$ be a Hilbert space such that $C^*(Z_n^{*m})\subset B(H)$. Then, by Arveson extension theorem (\cite{A}) $\alpha$ can be extended to a completely positive map $\tilde{\alpha}:B(H)\to M_d(\bC)$. It follow from Lemma 2.1 in \cite{St} that
\be \phi(a\otimes b)=\Tr(\tilde{\alpha}(a)b^\mathrm{t}),\qquad a\in B(H),\,b\in M_d(\bC)\ee 
defines some positive normal functional on $B(H)\otimes M_d(\bC)$. Hence, there is a density matrix $\rho$ acting on $H\otimes\bC^d$ such that
$\phi(t)=\Tr(\rho t)$ for $t\in B(H)\otimes M_d(\bC)$. Now, one can show that
\be \tilde{\alpha}(a)=\Tr_H(\rho(a\otimes\un_d)),\qquad a\in B(H).\ee 
Let $\iota:U_{m,n}\to C^*(\bZ_n^{*m})$ be the canonical embedding, and let $E_x^a=\iota(p_{ax})$. Then we have
\be \sigma(a|x)=\alpha(p_{ax})=\tilde{\alpha}(p_{ax})=\Tr_H(\rho(E_x^a\otimes\un_d)).\ee 
Hence, $\sigma\in\cQ_\mathrm{s}$.


\end{proof}

Let $(V_{m,n}\ot M_d(\bC))_1=\{(A_{xa}):\,\sum_a\Tr(A_{ax})=1\;\mbox{for any $x$}\}$.
Then we have
\begin{corollary}
We have the following equality
\be \cQ_s=(V_{m,n}\otimes_\mathrm{max} M_d(\bC))^+\cap (V_{m,n}\ot M_d(\bC))_1\ee 
\end{corollary}

\section{Bell and steering inequalities in operator systems framework}
One of the main tasks while studying nonlocality is to quantify the difference between the set $\cL$ and $\cQ$. Having the characterization of $\cQ$ given in section 4 one can ask whether is possible to find useful tools for this in the framework of operator systems.

Let $V$ be a real linear space and let $e\in V$ be a distinguished nonzero element. Moreover, let $C\subset V^\mathrm{d}$ be a cone so that $\phi(e)>0$ for every $\phi\in C$. Now, for any $v\in V$ let us define the number $N_C(v)$ by the formula
\begin{equation}
N_C(v)=\sup\left\{\dfrac{\phi(v)}{\phi(e)}:\phi\in C\right\}.
\end{equation}
Then we have
\begin{proposition}
Let $v,w\in V$ and $\alpha\in \bR$. Then
\begin{enumerate}
\item[(i)] $N_C(\alpha e)=\alpha$,
\item[(ii)] $N_C(\alpha v)=\alpha N_C(v)$,
\item[(iii)] $N_C(v+w)\leq N_C(v)+N_C(w)$.
\end{enumerate}
\end{proposition}
\begin{proof}
Direct calculations.
\end{proof}
Next, consider two cones $C_1\subset V^d$ and $C_2\subset V^d$. We would like to apply the above notion to compare sizes of this two cones. 
Firstly, let us formulate rather obvious 
\begin{lemma}
We have the inclusion $C_1\subset C_2$ if and only if $N_{C_1}(v)\leq N_{C_2}(v)$ for every $v\in V$. Moreover, if $N_{C_1}(v)<N_{C_2}(v)$ for some $v\in V$, then the inclusion is proper.
\end{lemma}
Now, let us propose the following definition
\begin{definition}
Let $C_1\subset C_2\subset V^d$. We say that an element \textit{$v\in V$ violates $C_1$ with $C_2$} if $0< N_{C_1}(v)<N_{C_2}(v)$. The number
\be LV_{C_1\subset C_2}(v)=\dfrac{N_{C_1}(v)}{N_{C_2}(v)}\ee 
will be called \textit{the largest violation of $v$.}
\end{definition}

Now, let us pass to the main example. Let $V=U_{m,n}\ot U_{m,n}$ and $e=\un$.
Let us observe, that Definition \ref{d:box}(2) leads to the following observation
\begin{proposition}
The set LHV correlation boxes is equal to the set of normalized functionals from $V_{m,n}^+\ot V_{m,n}^+$.
\end{proposition}
Thus, let $C_1=V_{m,n}^+\ot V_{m,n}^+$ and $C_2=(V_{m,n}\ot_\mathrm{max} V_{m,n})^+$. Obviously, $C_1\subset C_2$.
Following \cite{JP1,JP2} we formulate the following
\begin{definition}
A \textit{Bell inequality} (or \textit{Bell functional}) is a set of numbers $t=\{t_{abxy}:\,a,b=0,1,\ldots,m-1,\;x,y=1,2,\ldots,n\}$.
For such $t$ and every correlation box $(p(ab|xy))$ one can define 
\be \langle t,p\rangle=\sum_{a,b,x,y}t_{abxy}p(ab|xy).\ee 
\end{definition}
Let us observe that every Bell inequality $t=\{t_{abxy}\}$ determines some element
$\sum_{a,b,x,y}t_{abxy}p_{ax}\ot p_{by}\in U_{m,n}\ot U_{m,n}$ which will be denoted also by $t$. Now we can formulate the following
\begin{definition}
\textit{The largest violation} $LV(t)$ of a Bell inequality $t\in U_{m,n}\ot U_{m,n}$ is 
defined as
\be LV(t)=LV_{V_{m,n}^+\ot V_{m,n}^+\subset (V_{m,n}\ot_\mathrm{max}V_{m,n})^+}(t).\ee 
\end{definition}
\begin{remark}
The numbers $N_{V_{m,n}^+\ot V_{m,n}^+}(t)$ and $N_{(V_{m,n}\ot_\mathrm{max}V_{m,n})^+}(t)$ can be interpreted respectively as classical bound and quantum bound of $t$ (see \cite{JP2}).
\end{remark}

The same can be done for steering scenario (see \cite{Pus,YMH}). Namely, we can state the following
\begin{definition}
A \textit{steering inequality} (or \textit{steering functional}) is a set of matrices $F=\{F_{ax}\in M_d(\bC):\,a=0,1,\ldots,m-1,\;x=1,2,\ldots,n\}$.
For such $F$ and every assemblage $(\sigma(a|x))$ one can define 
\be \langle F,\sigma\rangle=\sum_{a,x}\Tr (F_{ax}\sigma(a|x)).\ee 
\end{definition}
As a steering functional $F$ can be viewed as element $\sum_{a,x}p_{ax}\ot F_{ax}\in U_{m,n}\ot M_d(\bC))$, 
the largest violation of a steering inequality can be defined as
\be LV(F)=LV_{V_{m,n}^+\ot M_d(\bC)^+\subset(V_{m,n}\ot_\mathrm{max}M_d(\bC))^+}(F).\ee 

\begin{example}
Let $n=m=2$. The space $V_{2,2}\ot V_{2,2}$ can be identified with the space $BS$ of $4\times 4$ matrices $(a_{ij})_{i,j=1,2,3,4}$ such that $a_{i1}+a_{i2}=a_{i3}+a_{i4}$ for every $i$, and $a_{1j}+a_{2j}=a_{3j}+a_{4j}$ for every $j$. Furthermore, entries of each matrix $(a_{ij})$ from $V_{2,2}^+\ot V_{2,2}^+$ have the form
$a_{ij}=\sum_\lambda q_\lambda b_i^\lambda c_j^\lambda$ for $i,j=1,2,3,4$, where $b_i$ and $c_j$ are numbers such that $b_1^\lambda+b_2^\lambda=b_3^\lambda+b_4^\lambda$ and $c_1^\lambda+c_2^\lambda=c_3^\lambda+c_4^\lambda$, and $q_\lambda\leq 0$ with $\sum_\lambda q_\lambda=1$. Let $t$ be a Bell inequality. It can be represented as $4\times 4$ matrix $(t_{ij})_{i,j=1,2,3,4}$.
Now, let us estimate the classical bound. For a normalized element $a=(a_{ij})\in V_{2,2}^+\ot V_{2,2}^+$ we have:
\begin{eqnarray}
|\langle t,a\rangle|&=& \Big|\sum_\lambda q_\lambda\sum_{i,j=1}^4b_i^\lambda c_j^\lambda t_{ij}\Big|\leq \sum_\lambda q_\lambda\Big|\sum_{i,j}b_i^\lambda t_{ij} c_j^\lambda\Big|\nonumber \\
&=&
\sum_\lambda q_\lambda\big|\langle b^\lambda|t|c^\lambda\rangle\big|\leq
\sum_\lambda q_\lambda \|b^\lambda\|_2\|t\|_\infty\|c^\lambda\|_2\leq 4\|t\|_\infty,
\end{eqnarray}
where $\|\cdot\|_2$ denotes usual euclidean norm, while $\|\cdot\|_\infty$ is the operator norm. The last inequality follows from the fact that $\|b^\lambda\|_2\leq \|b^\lambda\|_1=\sum_ib_i^\lambda=2$, and similarly for $c^\lambda$.

Now, let us choose a Bell inequality $t=(t_{ij})$, where $t_{ij}=1$ for every $i,j=1,2,3,4$.
For this Bell inequality we can improve the above estimation. Namely, one can easily show that $N_{V_{2,2}^+\ot V_{2,2}^+}(t)=4$. On the other hand, we use the characterization of positive elements in $(V_{2,2}\ot_\mathrm{max}V_{2,2})^+$ given in \cite[Proposition 6.8]{FKPT}. It says that for each such element $(a_{ij})$ there is $p\in\bN$ and matrices $A_i,B_j\in M_p(\bC)$ such that $a_{ij}=\Tr(A_iB_j)$ for $i,j=1,2,3,4$. Let $p=2$ and let us consider the following vectors from $\bC^2$:
\be e_1=(1,0),\quad e_2=(0,1),\quad f_1=\left(\frac{\sqrt{2}}{2},\frac{\sqrt{2}}{2}\right),\quad f_2=\left(-\frac{\sqrt{2}}{2},\frac{\sqrt{2}}{2}\right).\ee 
Define now
\ben
A_1=A_3=\frac{1}{2}|e_1\rangle\langle e_1|,\quad A_2=A_4=\frac{1}{2}|e_2\rangle\langle e_2|, \nonumber \\
B_1=B_3=\frac{1}{2}|f_1\rangle\langle f_1|,\quad B_2=B_4=\frac{1}{2}|f_2\rangle\langle f_2|.
\een
Then the determined element $(a_{ij})\in(V_{2,2}\ot_\mathrm{max}V_{2,2})^+$ is of the form $a_{ij}=\frac{\sqrt{2}}{8}$ 
for every $i,j=1,2,3,4$. Thus
\be \langle t,a\rangle =\sum_{i,j}t_{ij}a_{ij}=16\cdot\frac{\sqrt{2}}{8}=4\sqrt{2}.\ee 
Therefore, 
\be N_{(V_{2,2}\ot_\mathrm{max}V_{2,2})^+}(t)\geq 4\sqrt{2},\ee 
and consequently
\be LV(t)\geq\sqrt{2}.\ee 
\end{example}


{\bf Acknowledgemenets}
M.H. thanks Robert Alicki for discussion, while M.M. and Y.Z. thank Vern Paulsen for valuable remarks on tensor products of operator systems. The work is supported by
Foundation for Polish Science TEAM project co-financed by the EU European Regional Development
Fund, Polish Ministry of Science and Higher Education Grant no. IdP2011 000361, ERC AdG grant QOLAPS and  EC grant RAQUEL.
Part of this work was done in National Quantum Information Center of
Gda{\'n}sk. Part of this work was done when the authors attended the program “Mathematical
Challenges in Quantum Information” at the Isaac Newton Institute for Mathematical
Sciences, University of Cambridge.

\end{document}